\documentclass{amsart}[12pt]
\def\doctype{}

\usepackage{latexsym,amssymb}
\usepackage{color}
\usepackage{fancyhdr}
\usepackage{tikz}
\usetikzlibrary{plotmarks}
\usetikzlibrary{decorations.pathreplacing}
\usepackage{hyperref}

\def\cA{\mathcal{A}}
\def\cB{\mathcal{B}}

\def\F{{\mathbb F}}
\def\Z{{\mathbb Z}}

\newcommand{\comment}[1]{}

\numberwithin{equation}{section}

\fancypagestyle{titlepage}{
\fancyhead[R]{\doctype}
\fancyhead[CL]{}
\cfoot{\vspace{5pt} \thepage}
}

\let\oldsection\section
\newcommand\boldsection[1]{\oldsection{\bf #1}}
\newcommand\starsection[1]{\oldsection*{\bf #1}}
\makeatletter
\renewcommand\section{\@ifstar\starsection\boldsection}
\makeatother

\newtheoremstyle{thm}
  {12pt}		  
  {0pt}  
  {\sl}  
  {\parindent}     
  {\bf}  
  {. }    
  { }    
  {}     
\theoremstyle{thm}
\newtheorem{thm}{Theorem}[section]  
\newtheorem{lemma}[thm]{Lemma}     
\newtheorem{cor}[thm]{Corollary}

\newtheorem{cons}[thm]{Construction}

\newtheoremstyle{definition}
  {12pt}		  
  {0pt}  
  {}  
  {\parindent}     
  {\bf}  
  {. }    
  { }    
  {}     
\theoremstyle{definition}

\newcommand\rk{{\sc Remark.} }

\renewcommand{\proofname}{Proof}

\makeatletter
\renewenvironment{proof}[1][\proofname]{\par
  \pushQED{\qed}%
  \normalfont \partopsep=\z@skip \topsep=\z@skip
  \trivlist
  \item[\hskip\labelsep
        \scshape
    #1\@addpunct{.}]\ignorespaces
}{%
  \popQED\endtrivlist\@endpefalse
}
\makeatother

\makeatletter
\renewcommand*\@maketitle{%
  \normalfont\normalsize
  \@adminfootnotes
  \@mkboth{\@nx\shortauthors}{\@nx\shorttitle}%
  \global\topskip42\p@\relax 
  \@settitle
  \ifx\@empty\authors \else {\vskip 1em
\vtop{\centering\shortauthors\@@par}} \fi
  \ifx\@empty\@date \else {\vskip 1em \vtop{\centering\@date\@@par}}\fi 
  \ifx\@empty\@dedicatory
  \else
    \baselineskip18\p@
    \vtop{\centering{\footnotesize\itshape\@dedicatory\@@par}%
      \global\dimen@i\prevdepth}\prevdepth\dimen@i
  \fi
  \@setabstract
  \normalsize
  \if@titlepage
    \newpage
  \else
    \dimen@34\p@ \advance\dimen@-\baselineskip
    \vskip\dimen@\relax
  \fi
} 
\renewcommand*\@adminfootnotes{%
  \let\@makefnmark\relax  \let\@thefnmark\relax
  \ifx\@empty\@subjclass\else \@footnotetext{\@setsubjclass}\fi
  \ifx\@empty\@keywords\else \@footnotetext{\@setkeywords}\fi
  \ifx\@empty\thankses\else \@footnotetext{%
    \def\par{\let\par\@par}\@setthanks}%
  \fi
\thispagestyle{titlepage}
}
\makeatother

\begin{document}

\title[Bounded flats]{Pairwise balanced designs covered by bounded flats}

\author{Nicholas M.A.~Benson}

\author{Peter J.~Dukes}
\address{\rm 
Mathematics and Statistics,
University of Victoria, Victoria, Canada
}
\email{bensonn@uvic.ca, dukes@uvic.ca}

\thanks{Research of the authors is supported by NSERC}

\date{\today}

\begin{abstract}
We prove that for any $K$ and $d$, there exist, for all sufficiently large admissible $v$, a pairwise balanced design PBD$(v,K)$ of dimension $d$ for which all $d$-point-generated flats are bounded by a constant independent of $v$.  We also tighten a prior upper bound for $K = \{3,4,5\}$, in which case there are no divisibility restrictions on the number of points.  One consequence of this latter result is the construction of latin squares `covered' by small subsquares.
\end{abstract}


\maketitle
\hrule

\section{Introduction}

A \textit{pairwise balanced design} is a pair $(X,\cB)$, where $X$ is a set of points and $\cB \subseteq 2^X$ is a family of \emph{blocks} which cover every pair of different points exactly once.  The standard notation is PBD$(v,K)$, where $v = |X|$ and $K \subseteq \Z_{\ge 2}$ contains the allowed block sizes.
Closely related objects are Steiner systems, balanced incomplete block designs, and linear spaces.

Since the set of blocks incident with any point must contain each other point once, and since the set of pairs of points must partition into the pairs covered in each block, we have the `divisibility' conditions
\begin{align}
\label{local}
\alpha(K) & \mid v-1~~\text{and}\\
\label{global} 
\beta(K) & \mid v(v-1),
\end{align}
where $\alpha(K):=\gcd\{k-1: k \in K\}$ and $\beta(K):=\gcd\{k(k-1): k \in K\}$.  The integers $v$ satisfying (\ref{local}) and (\ref{global}) are \emph{admissible}.
An important theorem of R.M.~Wilson states that admissibility is suficient for existence of a PBD$(v,K)$, provided $v$ is large.

\begin{thm}[Wilson, \cite{BIB:Wilson}]
\label{THM:AsymptoticPBD}
There exist PBD$(v,K)$ for all sufficiently large admissible $v$.
\end{thm}
 
Let $(X,\cB)$ be a pairwise balanced design.  A \emph{flat} (or \emph{subdesign}) is a pair $(Y,\cB|_Y)$, where $Y \subseteq X$ and $\cB|_Y:=\{B \in \cB: B \subseteq Y\}$ have the property that any two distinct points in $Y$ are together in a unique block of $\cB|_Y$.  Flats in $(X,\cB)$ form a lattice under intersection.  As such, any set of points $S \subseteq X$ generates a flat $\langle S \rangle$ equal to the intersection of all flats containing $S$.  Alternatively, $\langle S \rangle$ can be computed algorithmically starting from $S$ by repeatedly including points on blocks defined by previously included pairs.

The {\em dimension} of a PBD is the maximum integer $d$ such that any set of $d$ points generates a proper flat. 
This definition, is taken from the context of linear spaces; see \cite{BIB:Delan}.
For example, the flat generated by any two points is the line containing them.  So every PBD$(v,K)$ with more than one block has dimension at least two.  

The classical geometries come with nontrivial dimension. Let $q$ be a prime power and $\mathbb{F}_q$ the finite field of order $q$.  We can take as point set
the vector space $X=\F_q^{d}$, and as flats all possible translates $x+W$ of subspaces 
$W$ in $X$.  This forms the {\em affine space} AG$_d(q)$.  Viewing the one-dimensional flats as blocks, we obtain a PBD$(q^d,\{q\})$ 
of dimension $d$.  For example, the case $q=3$, $d=4$ recovers the popular card game `Set'.

Likewise, the set of direction vectors $(\F_q^{d+1} \setminus \{ \mathbf{0} \})/\F_q^*$ induces the \emph{projective space} PG$_d(q)$.  There are $[d]_q:=1+q+\dots+q^d$ projective points in total.  Blocks are projective lines defined by two-dimensional subspaces of $\F_q^{d+1}$, and it is easy to see these have size $q+1$.  As indicated by the notation, the dimension of PG$_d(q)$ as a PBD is also $d$.

In the binary case, PG$_2(2)$ is the familiar PBD$(7,\{3\})$ or `Fano plane'.   Increasing the dimension, PG$_3(2)$ is a PBD$(15,\{3\})$ such that any three points are either collinear or define a Fano plane.  Teirlinck in \cite{BIB:Teirlinck} was a key early investigator of dimension in general Steiner triple systems ($K=\{3\}$).

For general $K$ and any desired minimum dimension $d$, there is a recent existence theory in the spirit of Theorem~\ref{THM:AsymptoticPBD}.

\begin{thm}[\cite{BIB:PrescribedMinimumDimension}]
Given $K \subseteq \Z_{\ge 2}$ and $d \in \Z_+$, there exists a PBD$(v,K)$ of dimension at least $d$ for all sufficiently large admissible $v$.
\end{thm}

In this paper, though, we are interested in a strengthening in which we (universally) bound all $d$-point-generated flats.  Here is the statement of our first main result in this direction.

\begin{thm}
\label{THM:Asymptotic}
For a given $K$ and $d$, there exists, for all sufficiently large admissible $v$, a PBD$(v,K)$ such that any $d$ points generate a flat of size at most $f(d, K)$, a constant independent of $v$.
\end{thm}

This is perhaps surprising at first glance, and accordingly there are technical challenges to overcome in the proof.  The necessary background and proof are covered in Sections \ref{SEC:GDDs} and \ref{SEC:AsymptoticsConstruction} to follow.

The set $K=\{3,4,5\}$ of block sizes  is of special interest to us.  First, the divisibility conditions disappear in this case, so that all positive integers $v$ are admissible.  (The only exceptions to existence are $v=2,6,8$.)  Second, idempotent quasigroups can be constructed from this $K$ by a simple gluing operation.  As we see later, universally bounded flats lead to some interesting extremal objects, such as latin squares covered by small subsquares and one-factorizations of the complete bipartite graph $K_{n,n}$ with only short cycles.  This motivates our second main result.

\begin{thm}
\label{THM:345}
There exist PBD$(v,\{3,4,5\})$ for all $v \neq 2,6,8$ such that any three points generate a flat of size at most 
$63$, unless $v-[e]_4 \in \{1,3,9\}$ for some integer $e \ge 3$, in which case any three points generate a flat of size at most $94$.
\end{thm}

This is proved in Section 4 by carefully truncating and inflating points in the projective spaces over $\F_4$.  In fact, our construction also upper-bounds all $d$-point-generated flats for arbitrary $d$, but (as expected) this bound grows with $d$.  To put the result in context, our bound of 63  is only three times the smallest nontrivial PBD$(v,\{5\})$, this being PG$_2(4)$.  It appears difficult to eliminate the sparse family of exceptions which escape this bound.  The first three, though, were settled in \cite{BIB:NiezenThesis}, where it was proved that a PBD$(v,\{3,4,5\})$ of dimension three exists for all $v \ge 48$.  Since proper flats in such a PBD have size less than $v/2$, we could actually write `$e \ge 4$' in Theorem~\ref{THM:345}.

\section{Group divisible designs}
\label{SEC:GDDs}

This section develops the necessary background for our proof of Theorem~\ref{THM:Asymptotic}. 

A \emph{group divisible design}, or GDD, is a triple $(X,\Pi,\cB)$, where $X$ is a set of \emph{points}, $\Pi$ is a partition of $X$ into \emph{groups} (there need not be algebraic structure), and $\cB$ is a set of \emph{blocks} such that 
\vspace{-11pt}
\begin{itemize}
\item
a group and a block intersect in at most one point; and
\item
every pair of points from distinct groups is together in exactly one block.
\end{itemize}
\vspace{-11pt}
Writing $T$ for the list of group sizes, we adopt the notation GDD$(T,K)$ for similarity with the notation for PBDs.  Typically, $T$ is called the \emph{type} of the GDD.  When $T$ contains, say, $u$ copies of the integer $g$, this is abbreviated with `exponential notation' as $g^u$.  If the type is just $g^u$ for some $g,u$, the resulting GDD is called \emph{uniform}.   A GDD$(1^v,K)$ is just a PBD$(v,K)$.  Another abbreviation we shall use is to write simply `$k$' instead of `$\{k\}$' in the notation.  With this in mind, a GDD$(n^k,k)$ is equivalent to a set of $k-2$ mutually orthogonal latin squares of order $n$, where two groups index rows and columns, and each other group defines a square. 

Simple counting reveals the necessary divisibility conditions 
\begin{equation}
\label{nec-gdd}
k-1 \mid g(u-1)~~\text{and}~~k(k-1) \mid g^2 u(u-1)
\end{equation}
on GDD$(g^u,k)$.  We now cite two useful asymptotic existence results, one for each parameter.

\begin{thm}[\cite{BIB:AsymptoticGDD}]
\label{THM:TallGDD}
Given integers $u \ge k \ge 2$, there exists a GDD$(g^u,k)$ for all sufficiently large integers $g$ satisfying \eqref{nec-gdd}.
\end{thm}

\begin{thm}[\cite{BIB:Chang}]
\label{THM:WideGDD}
Given $k$ and $g$, there exists a GDD$(g^u,k)$ for all sufficiently large integers $u$ satisfying \eqref{nec-gdd}.
\end{thm}

In fact, there is a version of Theorem~\ref{THM:WideGDD} for multiple block sizes.

\begin{thm}[\cite{Draganova,Liu}]
\label{asym-gdd}
Given $g$ and $K \subseteq \Z_{\ge 2}$, there exists a GDD$(g^u,K)$ for all
sufficiently large $u$ satisfying
\begin{align}
\label{local-gdd}
\alpha(K) & \mid g(u-1) ~~\text{and}\\
\label{global-gdd}
\beta(K) & \mid g^2 u(u-1).
\end{align}
\end{thm}

It is helpful to think of GDDs as `holey' PBDs, in the sense that groups of a GDD can be `filled' with appropriately-sized PBDs.

\begin{cons}[Filling groups]
\label{CON:fill}
Suppose there exists a GDD$(T,K)$ on $v$ points.\\
{\rm (a)} If, for each group size $g$ in $T$, there exists a PBD$(g,K)$, then there exists a PBD$(v,K)$.\\
{\rm (b)} If, for each group size $g$ in $T$, there exists a PBD$(g+1,K)$, then there exists a PBD$(v+1,K)$.
\end{cons}

\rk
In (b) above, we add a new a point at which every filled PBD intersects.  More generally, this point can instead be a common flat (say of size $h$) in each PBD, resulting in a PBD$(v+h,K)$.

Another feature of GDDs is that their groups admit a natural `inflation'.

\begin{cons}[Wilson's fundamental construction, \cite{ConsUses}]
\label{CON:WFC}
Suppose there exists a GDD $(X,\Pi,\cB)$, where $\Pi=\{X_1,\dots,X_u\}$.
Let $\omega:X \rightarrow \Z_{\ge 0}$, assigning nonnegative weights to each
point in such a way that for every $B \in \cB$ there exists a
GDD$([\omega(x) : x \in B],K)$.  Then there exists a GDD$(T,K)$, where
\begin{equation}
\label{WFC-output}
T=\left[\sum_{x \in X_1} \omega(x),\dots,\sum_{x \in X_u}
\omega(x)\right].
\end{equation}
\end{cons}

The idea in the above construction is that points of the original `master' GDD get weighted, and blocks get replaced by small `ingredient' GDDs.  There is one noteworthy special case.  A weighting with $\omega(x)=0$ or $1$ for all $x \in X$ is called a \emph{truncation}; in this case, blocks get replaced by smaller blocks.  A careful truncation has a mild (or possibly no) effect on the set of allowed block sizes $K$.

It is important for our purposes to extend the notion of flats to GDDs, and in particular to analyze the impact of the preceding constructions on them.  Given a GDD, say $(X,\Pi,\cB)$, a \emph{sub-GDD} (or \emph{subdesign}) is a triple $(Y,\Pi_Y,\cB_Y)$, where $\Pi_Y$ is the restriction of $\Pi$ to $Y$ and $\cB_Y:=\{B \in \cB: B \subseteq Y\}$.  This is a natural extension of the definition for pairwise balanced designs (in which $\Pi$ consists of singletons).  As before, single blocks define subdesigns.  And now, if all points of $Y$ belong to the same group, they trivially define a sub-GDD with no blocks.

Construction~\ref{CON:WFC} in a sense preserves sub-GDDs.  When a sub-GDD of the master is weighted, it becomes (by a smaller application of the the same construction) a GDD whose type is as in (\ref{WFC-output}), except where the summations restrict to $x \in X_1 \cap Y$, etc.  The following is now clear.

\begin{lemma}
Suppose a GDD $(X,\Pi,\cB)$ has the property that any $d$ points is contained in a sub-GDD touching at most $t$ groups.  Then the result of applying Wilson's fundamental construction, regardless of the weights or ingredient GDDs, has the same property.
\end{lemma}

We need to be a little more careful with Construction~\ref{CON:fill}.  If a sub-GDD $Y$ of the input GDD has at least two (but not all) points from the same group, then filling this group with a PBD can cause $Y$ to no longer be a sub-GDD.  The remedy is to prefer subdesigns which intersect each group nicely.

\begin{lemma}
Consider a GDD $(X,\Pi,\cB)$ with a sub-GDD $Y$.  If each group $X_i \in \Pi$ is filled with a PBD $(X_i,\cA_i)$ such that $Y \cap X_i$ is a flat of $(X_i,\cA_i)$, then $Y$ is a flat of the PBD $(X,\cA \cup \cB)$, where $\cA:= \cup_i \cA_i$.  Likewise, if a point $\infty$ is added and each group $X_i$ is filled with a PBD $(X_i \cup \{\infty\},\cA_i')$ such that $Y \cap X_i$ is a flat of $(X_i,\cA_i')$, then $Y$ is a flat of the PBD $(X \cup \{\infty\},\cA' \cup \cB)$, where $\cA':=\cup_i \cA_i'$.  
\end{lemma}

\rk 
In our applications, we use this with sub-GDDs $Y$ intersecting the partition trivially, so that $Y \cap X_i = X_i$ or $\emptyset$ for each group.
 
In what follows, it is helpful to adopt interval notation for integers, with $[a,b]:=\{x \in \Z : a \le x \le b\}$.  Also, for sets $A,B \subset Z$, we write
$A+B:=\{a+b:a \in A, b \in B\}$, as well as $x+A=\{x+a: a \in A\}$.

\section{Proof for general block sizes}
\label{SEC:AsymptoticsConstruction}

The broad idea of the proof of Theorem~\ref{THM:Asymptotic} is to carefully tinker with the affine space AG$_{e}(q)$, $e>d$, which has the property of bounded $d$-point-generated flats but lacks the generality in its parameters.  This is similar in spirit to the approach used in \cite{BIB:PrescribedMinimumDimension}.

In a little more detail, we construct a variety of ingredient non-uniform GDDs, mostly arising from truncation of uniform GDDs.  Then, we weight the points of AG$_\delta(q)$ according to the ingredients, applying Wilson's fundamental construction to construct a large GDD based on the affine space.  We finish the proof by filling groups of this GDD.  Care must be taken to ensure that the group sizes do not exceed a universal bound.

We first present the needed non-uniform GDDs.

\begin{lemma}
\label{ING:one}
Given integers $u \ge k \ge 4$, there exists a GDD$(g^i(g-1)^{u-i},\{k-2,k-1,k\})$ for all $g \gg 0$ and all $i \in [0,u]$.  Moreover, we may assume that each group has some point incident only to blocks of size $k-1$ and $k$.
\end{lemma}
\begin{proof}
Given $g \gg 0$, choose $h \in [g,g+k(k-1)]$ with $k(k-1) \mid h$.  We may assume, by Theorem~\ref{THM:TallGDD}, that there exists GDD$(h^u,k)$.

We desire to truncate $\Delta:=h-g$ points from $i$ groups and $\Delta+1$ points from the remaining $u-i$ groups in such a way that at most two points get removed from each block, and also such that some point in each group is incident to only blocks with at most one point removed.
This is straightforward via an iterative random process, since $h$ is large relative to $u$ and $k$.  Selecting up to $\Delta+1$ points from each of two groups renders at most $(k-2) (\Delta+1)^2 < k^5$ points in blocks reduced by two points.  We `protect' these points from truncation in later steps, and choose another group for truncation, and enlarge the protected set.  There are $O(u^2 k^5)$ protected points throughout, and so for large $g$ the desired truncation is possible.  Moreover, we may assume each group has unprotected points at the end, and these points are not on a block with any removed pair.
\end{proof}

\begin{lemma}
\label{ING:two}
Given integers $u \ge l \ge 2$, there exists a GDD$(m^i (m+l)^{u-i-1} (m+x)^1,\{l,l+1\})$ for all $m \gg 0$ with $l \mid m$, all $i \in [0,u-1]$ and all $x \in [0,l]$.
\end{lemma}

\begin{proof}
Choose $k \ge l+2$ large enough so that there exist, by Theorem~\ref{THM:WideGDD}, both GDD$(l^{k-1},\{l+1\})$ and GDD$(l^k,\{l+1\})$.  From the former GDD, we may truncate all points of one group, resulting in a GDD$(l^{k-2},\{l,l+1\})$.

Give weight $l$ to all points of the GDD from Lemma~\ref{ING:one}.  After Construction~\ref{CON:WFC}, the result is a GDD$((gl)^i((g-1)l)^{u-i},\{l-1,l\})$.  This proves the lemma for $x=l$.  

The additional property of the GDD in Lemma~\ref{ING:one} ensures that some weighted point is incident only with blocks of size $l+1$.  So, after weighting, if we truncate from the resulting set of $l$ points, leaving $x$ behind, all block sizes remain in $\{l,l+1\}$.  This allows us to reduce one group size from $m+l$ to $m+x$.
\end{proof}

We now return to our main goal for this section.

\begin{proof}[Proof of Theorem~\ref{THM:Asymptotic}]
Let $\alpha:=\alpha(K)$, $\beta:=\beta(K)$, and put $\gamma:=\beta/\alpha$.  Let $R$ denote the set of integers $r$ such that there exist PBD$(\alpha r + 1,K)$.  By Theorem~\ref{THM:AsymptoticPBD} and a calculation, there is an integer $r_0$ such that $r \in R$ for all $r \ge r_0$ satisfying $\gamma \mid r(\alpha r +1)$.

Fix some $b \equiv 0 \pmod{\beta}$ large enough so that there exist, by Theorem~\ref{THM:WideGDD}, both GDD$(\alpha^b,K)$ and GDD$(\alpha^{b+1},K)$.  
Put $q \equiv 1 \pmod{b}$ a prime power and consider the affine space AG$_e(q)$, $e>d$.

Given the above $b,q$, there exists, by Lemma~\ref{ING:two}, a GDD$(m^i (m+b)^{q-i-1} (m+x)^1,\{b,b+1\})$ for all $m \gg 0$ with $b \mid m$, all $i \in [0,q-1]$ and all $x \in [0,b]$.  We may further assume that $m \ge r_0$.
The result of applying Construction~\ref{CON:WFC} to AG$_e(q)$ with these weights and ingredients is a  
\begin{equation}
\label{eqGDD}
\text{GDD}(m^j (m+b)^{q^e-j-1} (m+x)^1,\{b,b+1\}),
\end{equation}
where $j$ takes on any value in $[0,q^e-1]$ and $x$ takes on any value in $[0,b]$.  From the underlying affine structure, every set of $d$ points in (\ref{eqGDD}) is contained in a sub-GDD intersecting at most $q^d$ groups, and we can assume the intersection with the group partition is trivial.

Next, apply Construction~\ref{CON:WFC} again, this time to (\ref{eqGDD}) with constant weight $\alpha$.  The result is a
GDD$((\alpha m)^j (\alpha(m+b))^{q^e-j-1} (\alpha(m+x))^1,K)$, and subdesigns in (\ref{eqGDD}) have been inflated by $\alpha$.

To finish off, add a point apply Construction~\ref{CON:fill}, noting that $m,m+b \in R$ and, when admissible, we also have $m+x \in R$.
The result is a PBD$(\alpha (m q^e + b(q^e-j-1)+x)+1,K)$.

Observe that every sufficiently large integer in $R$ is expressible as $m q^e + b(q^e-j-1)+x$ for some $x \in [0,b]$ with $\gamma \mid x(\alpha x+1)$, some $j \in [0,q^e-1]$, some $m \in [r_0,qr_0]$ with $b \mid m$, and some $e > d$.  In this case we have that every set of $d$ points in our PBD is contained in a flat of size at most $f(d,K)=\alpha q^d(qr_0+b)+1$.
\end{proof}

\section{The case $K=\{3,4,5\}$}

Recall that our second main result, Theorem~\ref{THM:345}, asserts an explicit upper bound of 94 (often 63) on three-point-generated flats in some PBD$(v,\{3,4,5\})$.
We divide the proof into two separate cases (these define Subsections 4.1 and 4.2 to follow) depending on ranges of values of $v$.   Both arise from applying Construction~\ref{CON:WFC} to PG$_d(4)$, which recall is a PBD$(v,\{5\})$ for $v=[d]_4=\sum_{i=0}^d 4^i$.  In PG$_d(4)$, any three non-collinear points are contained in a projective plane on 21 points.  

We choose $d \ge 3$ so that $[d]_4$ is nearby $v$.  First, the `truncation case', exclusively uses weights 0 and 1 to cover the range $3[d-1]_4 \le v \le [d]_4$.  In this case, any three points remain in a flat on at most 21 points.  Next, the `inflation case' uses weights 3 and 4 to treat the range $[d]_4 \le v  \le 3[d]_4$.  This is the case where our larger flat sizes occur.  The dimension $d$ gets incremented and intervals overlap.

\subsection{Truncation}

Here, we show that PG$_d(4)$ can be truncated to a PBD$(v,\{3,4,5\})$ for any $v \in [3[d-1]_4, [d]_4]$.  
Alternatively, we would like to keep $v$ points of PG$_d(4)$ such that no line contains exactly two points.  (Lines with 0 or 1 point get discarded.) 
We call a truncation of PG$_d(4)$ \emph{legal} if it avoids leaving any line of size two.

Let us review the structure of PG$_d(q)$ in a little more detail.  Every PG$_d(q)$ contains several hyperplanes (copies of PG$_{d-1}(q)$) as maximal proper flats.  Truncation of a hyperplane results in AG$_d(q)$.  Every copy of PG$_{d-2}(q)$ in PG$_d(q)$ is the intersection of $q+1$ hyperplanes.  Truncating the intersection results in $q+1$ affine spaces AG$_{d-1}(q)$.  Let us call these \emph{pages}, with respect to the choice of codimension-two space PG$_{d-2}(q)$, which we call the \emph{spine}.

We are interested in $q=4$.  (Incidentally, this case also admits `Baer subspaces', but these are not needed in what follows.)  
Take a fixed spine in PG$_d(4)$ and consider its five pages.  With respect to this partition, there are three classes of lines (See also Figure~\ref{PG-structure}):
\begin{itemize}
\item [(A)] lines which have four points in some  page and one point in the spine;
\item [(B)] lines contained entirely in the spine; and
\item [(C)] lines which touch each page once.
\end{itemize}

\begin{figure}
\caption{spine/page structure and line classes in PG$_d(4)$}
\hspace{4cm}
\begin{tikzpicture}[darkstyle/.style={circle,draw,fill=gray!40,minimum size=20}]

\def\colgreen{white!10!green};
\def\colred{white!25!red};
\def\colblue{white!25!blue};
\def\textgreen{green};
\def\textred{red};
\def\textblue{blue};

\foreach \angle in {0,72,144,216,288}
{
\draw[rounded corners=4pt,rotate around={\angle-18:(0,0)}] (0.6,-0.42) rectangle (2.2,0.42);
\filldraw[rotate around={\angle-18:(0,0)},color=\colgreen] (1.0,0.2) circle [radius=.1];
}

\draw (0,0) circle [radius=.5];
\draw[color=\colgreen] (0,0) circle [radius=1.0];

\draw[color=\colblue] plot coordinates{(0.20,0.25) (0.20+1,0.25+1.376)};

\foreach \run in {0,0.25,0.5,0.75,1.0}
{
\filldraw[color=\colblue] (0.20+\run,0.25+1.376*\run) circle [radius=.1];
}

\draw[color=\colred] plot coordinates{(-0.3,-0.1) (0.3,-0.1)};

\foreach \run in {0,0.15,0.3,0.45,0.60}
{
\filldraw[color=\colred] (\run-0.30,-0.1) circle [radius=0.05];
}

\def\unit{0.5};
\def\base{1.6};
\def\basex{2.4};

\node[anchor=north west,text width=6cm] (lineA) at (\basex,\base) {Line Types:};
\node[anchor=north west,text width=6cm,color=\textblue] (lineA) at (\basex,\base-\unit) {Type A};
\node[anchor=north west,text width=6cm,color=\textred] (lineA) at (\basex,\base-2*\unit) {Type B};
\node[anchor=north west,text width=6cm,color=\textgreen] (lineA) at (\basex,\base-3*\unit) {Type C};


\end{tikzpicture}
\label{PG-structure}
\end{figure}
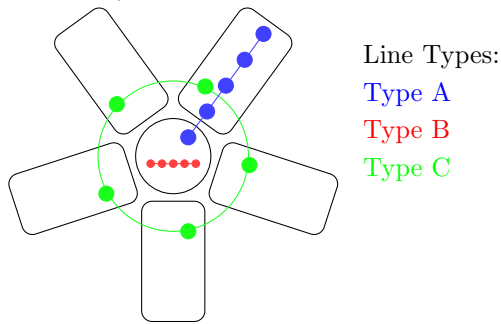

A key observation is that two legally truncated hyperplanes can be glued along their spine.  

\begin{lemma}
\label{LEM:Gluing}
Suppose PG$_{d-1}(4)$ can be legally truncated by either $n_1$ or $n_2$ points, in both cases with some hyperplane left untouched.  Then PG$_d(4)$ can be truncated by $n_1+n_2$ points.  Moreover, suppose legal truncations of PG$_{d-1}(4)$ as above each remove precisely some copy of PG$_i(4)$ from one hyperplane.  Then PG$_d(4)$ can be truncated by $n_1+n_2-[i]_4$ points.
\end{lemma}
\begin{proof}
Consider the spine/page structure of PG$_d(4)$.  In the first case, truncate $n_1$ and $n_2$ points from two pages, leaving the spine and all other pages untouched.  In the second case, truncate similarly, except with a common PG$_i(4)$ truncated from the spine.

By hypothesis, no type (A) line reduces to size 2.  Since the spine has lost a flat, type (B) lines are truncated to either 0, 4, or 5 points.
Finally, since only two pages experience truncation, lines of type (C) have lost at most 2 points.
\end{proof}

\rk
Note that the output of the above construction leaves either an untouched or PG$_e(4)$-truncated hyperplane (actually three such hyperplanes). 

We now analyze the possibilities arising from this gluing operation. 

For $i \in [0,d-1]$, let $T(d,i)$ be the set of numbers of points we can legally truncate from PG$_d(4)$ while leaving a hyperplane with precisely some PG$_i(4)$ truncated.  Let $T(d) = \cup_{i=0}^{d-1} T(d,i)$.  Our goal is to show $T(d)$ contains the first several positive integers.

Put $S(d,i):=T(d,i)-[i]_4$ for $i \ge 0$.  For example, $S(1,0)=\{0,1,3,4\}$.  Note that we may legally truncate an entire AG$_d(4)$ outside of a hyperplane, so in particular $4^d \in S(d,i)$ for each $i$.
By Lemma~\ref{LEM:Gluing} and the remark following it, we also have that
\begin{equation}
\label{EQN:S1}
S(d,i) \supseteq S(d-1,i) + S(d-1,i)
\end{equation}
for $0 \le i \le d-2$.

\begin{lemma}
\label{LEM:PGinS}
For all positive integers $d$, we have $[d-1]_4 \in S(d,0)$.
\end{lemma}
\begin{proof}
The claim is true for $d=1$, since $1 \in S(1,0)$. Assume the claim is true for $d \ge 1$.  Since $4^d \in S(d,0)$, equation (\ref{EQN:S1}) gives
$4^d+[d-1]_4 = [d]_4 \in S(d+1,0)$.  The result follows by induction.
\end{proof}

We now consider $S(d,d-1)$ as a special case.  Recursively adding either four empty pages or one empty and three full pages, we have
\begin{equation}
\label{EQN:S2}
S(d,d-1) \supseteq \{0,3 \times 4^{d-1}\} + S(d-1,d-2).
\end{equation}

\begin{lemma}
\label{LEM:4times}
For all positive integers $d$, we have $[0,4^{d+1}] \subseteq 4 * S(d,d-1) \subseteq S(d+2,d-1)$.
\end{lemma}
\begin{proof}
The first containment is an easy induction.  We have $4 *S(1,0)=[0,16]$ and, for $d \ge 1$, $$4*S(d+1,d) \supseteq 4*\{0,3 \times 4^d\}+[0,4^{d+1}] =[0,4^{d+2}],$$
where  (\ref{EQN:S2}) has  been used.  The second containment is just two applications of (\ref{EQN:S1}).
\end{proof}

There are other possible recursive truncations, but the above are enough for our purposes.

\begin{proof}[Proof of Theorem~\ref{THM:345}, truncation case]
Let $d \ge 3$.  We prove here that PG$_d(4)$ admits a legal truncation to $v$ points when $3[d-1]_4 \le v \le [d]_4$.  It suffices to prove $[1,[d-1]_4+1] \subseteq T(d)$; this is done in two stages.

{\sc Low values.} $[1,\frac{1}{2} [d-1]_4] \subseteq T(d)$.\\
We show by induction that $A:=[0,\frac{1}{2} [d-1]_4] \subseteq S(d,0)$.  First, this is true for $d=1$ and $d=2$ since $0 \in S(1,0)$ and $0,1,2 \in S(2,0)$.  Assume the statement holds for $d$.  By (\ref{EQN:S1}) and Lemma~\ref{LEM:PGinS}, each of the sets $A+A, [d-1]_4+A, 2 \times 4^{d-1} + A$ is contained in $S(d+1,0)$.  It is easy to check that these intervals cover $[0,\frac{1}{2}[d]_4]$.

{\sc High values.} $[[d-3]_4,[d-1]_4+1] \subseteq T(d)$.\\
By Lemma~\ref{LEM:4times}, we have $[0,4^{d-1}] \subset S(d,d-3)$.  We also have $4^{d-1}+[1,5] \subset S(d,d-3)$, since each summand is also in $S(d-1,d-3)$.  Repeated use of (\ref{EQN:S1}) gives
\begin{equation}
\label{EQN:series}
4^{d-2}+4^{d-3}+\cdots+4^{i+2} + [0,4^{d-1}+5] \subseteq S(d,i)
\end{equation}
for $0 \le i < d-3$.  Let us denote the interval on the left of (\ref{EQN:series}) by $B_i$.  Since $[i]_4+B_i \subseteq T(d,i)$, it follows that $T(d,i)$ have overlapping intervals covering between $[d-3]_4$ and 
$$[0]_4+\max B_0 = 4^{d-1}+4^{d-2}+\cdots +4^2+6 = [d-1]_4+1.$$
The low and high values overlap, since $[d-3]_4 < \frac{1}{2}[d-1]_4$.
\end{proof}

\subsection{Inflation}

We require some specific GDDs which are easy to construct from small planes. 
See \cite{BIB:NiezenThesis,BIB:Teirlinck} for more details.
These play a similar role as the GDDs in Theorem~\ref{ING:one} for general $K$.   

\begin{lemma}
\label{LEM:smallGDD}
There exist GDD$(1^i 3^{5-i},\{3,4,5\})$ and GDD$(4^i 5^{5-i},\{3,4,5\})$ for all $i \in [0,5]$.  There also exist GDD$(1^i 3^{4-i},\{3,4\})$ for $i=0,1,4$.
\end{lemma}

Now, we simply inflate PG$_d(4)$ and replace with these ingredients.

\begin{proof}[Proof of Theorem~\ref{THM:345}, inflation case]
First, apply Construction~\ref{CON:WFC} to PG$_d(4)$ with weights $1$ and $3$.  By Lemma~\ref{LEM:smallGDD}, all needed ingredients exist.  Fill groups of size three with blocks.  The result is a PBD with block sizes in $\{3,4,5\}$ such that every three points is contained in a flat on at most $3 \times 21 = 63$ points..  The number of points in this construction hits all odd values from $[d]_4$ to $3[d]_4$.

Next, truncate one point from PG$_d(4)$, leaving a set $\mathcal{L}$ of $[d-1]_4$ disjoint lines of size four.  Carefully assign weights 1 and 3 to the remaining points so that lines in $\mathcal{L}$ have 0, 3 or 4 points of weight 3.   It is not possible to triple exactly 1, 2 or 5 points, but any other positive integer is a sum of $3$s and $4$s.  The needed ingredients for Construction~\ref{CON:WFC} again exist by Lemma~\ref{LEM:smallGDD}.   We achieve all even $v$ from $[d]_4+1$ to $3[d]_4-1$, except
for $[d]_4+\{1,3,9\}$.

For these remaining values, we work from PG$_{d-1}(4)$.  Give either $2$, $4$ or $10$ points weight 5 and the remaining points weight 4.  As before, replace weighted lines with the ingredients in the lemma and fill groups with blocks.  The result is a construction for any number of points in $4[d-1]_4+\{2,6,10\}=[d]_4+\{1,3,9\}$.  In this case, any three points belongs to a flat on at most $4 \times 21 + 10=94$ points.
\end{proof}

\rk
For only mildly large $d$, it is possible to find sets of 10 points in PG$_d(4)$, no four of which are coplanar.  In this case, we can replace `94' by `87'.

\subsection{Summary} 

The following table summarizes the upper bounds on three-point-generated flat sizes for $K=\{3,4,5\}$. 

\begin{center}
\begin{tabular}{lcr}
\hline
value(s) of $v$ &  comments &  flat bound \\
\hline
$[d]_4+\{1,3,9\}$ & worst-case & $94$   \\
$[d]_4+\{1,3,9\}$ & for large $d$ & $87$   \\
$[[d]_4,3[d]_4]$ & inflate, except as above & $63$ \\
$[3[d-1]_4,[d]_4]$ & truncate from PG$_d(4)$  & $21$ \\
$4^d$ & AG$_d(4)$ & 16 \\
$[d]_3$ & PG$_d(3)$ & $13$ \\
$3^d$ & AG$_d(3)$ & 9 \\
$[d]_2=2^{d+1}-1$ & PG$_d(2)$& 7 \\
\hline
\end{tabular}
\end{center}
In the case of AG$_d(4)$ and PG$_d(3)$, we actually have $K=\{3,4\}$ and a (sparse) class of legal truncations is possible.  See \cite{BIB:LinearSpacesSmall} for more details.

\section{Discussion and Applications}

Recall that a \emph{latin square} of order $n$ is an $n \times n$ array on $n$ symbols such that every row and every column exhausts the symbols (with no repetition).  Latin squares are equivalently the operation table of finite `quasigroups', which have a binary operation with two-sided cancellation laws.  A (latin) \emph{subsquare} is a sub-array which is itself a latin square. Note that such a sub-array need not be on a contiguous set of rows and columns.

In a latin square, we often assume the set of symbols (and row/column indices) is $[n]:=\{1,\dots,n\}$.   A latin square is \emph{idempotent} if the entry in diagonal cell $(i,i)$ is $i$ for each $i \in [n]$.  Idempotent latin squares exist for all $n \neq 2$  can be `glued along the diagonal' using a PBD.  In more detail, suppose we have a PBD$(n,K)$, where $K \subseteq \Z_{\ge 3}$.  For every block $B$, let $L^{B}$ be an idempotent latin square on the symbols of $B$. Then we obtain an $n \times n$ idempotent latin square $L$, defined by
\begin{equation}
\label{EQN:ls-glue}
L_{ij} = 
  \begin{cases}
    i, & \text{if $i=j$};\\
    L_{ij}^{B}, & i\neq j, \text{letting $B$ be the block for which $\{i,j\} \subset B$}.
  \end{cases}
\end{equation}
We then have the following direct consequence of Theorem~\ref{THM:345}.

\begin{cor}
For any positive integer $n$, there exists an $n \times n$ latin square with the property that any choice of cell and symbol appear together in a latin subsquare of size at most $94$.
\end{cor}

\begin{proof}
Take a PBD$(n,\{3,4,5\})$ coming from Theorem~\ref{THM:345} and construct an idempotent latin square based on it as in (\ref{EQN:ls-glue}).  The choice of a row, column, and symbol amounts to a selection of three points in the PBD.  Since these three points are contained in a flat of size at most 94, it follows that the chosen cell and symbol are together in a subsquare of at most this size.
\end{proof}

Latin squares are equivalent to one-factorizations (or proper $n$-edge-colorings) of the complete bipartite graphs $K_{n,n}$.  Given two factors (color classes), the union of their edges induces a bipartite $2$-factor.  The problem of minimizing, over all such factorizations, the largest component in any such 2-factor was posed by H\"aggkvist and studied in some recent papers.  In \cite{BIB:LinearSpacesSmall}, it was observed that the above gluing construction with PBD$(v,\Z_{\ge 3})$ also leads to an upper bound of $2 (\max |Y| -  \min |B|)$, where $Y$ is a three-point-generated flat and $B$ is a block.  A crude  bound of 1716 was obtained, but with $|Y| \le 94$ we can now do much better.

\begin{cor}
For any positive integer $n$, there exists an $n$-edge-coloring of $K_{n,n}$ with the property that all two-colored cycles have size at most $182$.
\end{cor}

Perhaps the truth is as low as $6$ for large $n$, so that a mix of 4-cycles and 6-cycles occur in any pair of distinct color classes.  The bound of 182 could be lowered further with a more sophisticated (and, preferably, cleaner) truncation strategy, improving our Subsection 4.1.  This may be a problem of geometric interest in its own right.

\end{document}